\documentclass[10pt]{amsart}

\usepackage{algorithm}
\usepackage{algpseudocode}
\usepackage{amsaddr}
\usepackage{amssymb}
\usepackage{bbm}
\usepackage{caption}
\usepackage{float}
\usepackage{graphicx}
\usepackage{geometry} 
\usepackage{hyperref}
\usepackage{mathtools}
\usepackage{setspace}
\usepackage{subcaption}
\usepackage[table]{xcolor} 

\usepackage{verbatim}         

\newcommand{\mR}{\mathbb{R}}
\newcommand{\mN}{\mathbb{N}}

\newcommand{\supp}{\mathrm{supp}\,}
\newcommand{\partialytau}{\partial_y\partial_{\tau}}
\DeclareMathOperator\artanh{artanh}

\newtheorem{thm}{Theorem}[section]
\newtheorem{lem}[thm]{Lemma}

\newtheorem{rem}[thm]{Remark}

\newtheorem{prop}[thm]{Proposition}
\newtheorem{assm}[thm]{Assumption}

\DeclareMathOperator{\Dom}{Dom}

\definecolor{bl}{rgb}{0 0 0}

\definecolor{ilariablue}{rgb}{0 0.4 0.85}
\definecolor{ilariared}{rgb}{0.95 0.1 0}

\definecolor{sjoerdgreen}{rgb}{0 0.7 0.2}

\definecolor{dmitrycol}{rgb}{0.9 0.2 0.4}

\title[Exponential time-decay for the wave equation]{On the exponential time-decay for the one-dimensional wave equation with variable coefficients} \author{Anton Arnold$^{1*}$, Sjoerd Geevers$^{2*}$, Ilaria
  Perugia$^{2*}$, Dmitry Ponomarev$^{1,3*}$}
\address{{\small
$^1$ Institute of Analysis and Scientific Computing, Vienna
University of Technology\\
Wiedner Hauptstrasse 8-10, 1040 Vienna, Austria\\
$^2$ Faculty of Mathematics, University of Vienna\\
Oskar-Morgenstern-Platz 1, 1090 Vienna, Austria\\
$^3$ St. Petersburg Department of V. A. Steklov Mathematical Institute, RAS,\\
Fontanka 27, 191023 St. Petersburg, Russia
}}
\thanks{*A. Arnold, S. Geevers, and I. Perugia have been funded by the
Austrian Science Fund (FWF)
through the project F~65 ``Taming Complexity in Partial Differential
Systems''. I. Perugia has also been funded by the FWF through the
project P~29197-N32. A. Arnold and D. Ponomarev were supported by the bi-national FWF-project I3538-N32}

\begin{document}


\maketitle

\begin{abstract}
We consider the initial-value problem for the
  {one-dimensional,} time-dependent wave equation with {positive,} 
  Lipschitz {continuous} coefficients, which are 
  constant outside a bounded region.
  Under the assumption of compact
  support of the initial data, we prove {that} the local energy
  decay{s {exponentially} fast in time,} and
  provide the explicit constant to which the solution converges for
  large times.
  We {give} explicit estimates of the rate of this exponential decay
  by two different techniques{. The first one is based on the
definition of a modified, weighted local energy, {with} suitably
constructed weights. The second one is based on the
integral formulation of the {problem}
and, under a more restrictive
assumption on the variation of the coefficients, allows us to obtain
improved decay rates.}
\end{abstract}
\medskip

{\footnotesize
}
\medskip

\section{Introduction}
\label{sec:intro}
The paper is dedicated to the investigation of the time-decay of the solution of the following initial-value problem 
\begin{equation}\label{eq:wave_1D}
\begin{cases}
\beta\left(x\right)\partial_{t}^{2}u\left(x,t\right)-\partial_x\left(\alpha\left(x\right)\partial_x u\left(x,t\right)\right)=0, & x\in\mathbb{R},\:t>0,\\
u\left(x,0\right)=u_0\left(x\right),\hspace{1em}\partial_{t}u\left(x,0\right)=u_1\left(x\right), & x\in\mathbb{R},
\end{cases}
\end{equation}
where the coefficients $\alpha$, $\beta$ are {positive,}
{Lipschitz}{, and constant outside a bounded domain,}
and the initial data $u_0$, $u_1$ are compactly supported.
Studying such time-decay is important for practical applications 
and it shall be used in our follow-up paper \cite{paper-LAP}.

Most relevant works, such as the recent 
ones \cite{Shapiro} and \cite{Boucl-Burq}, are focussed on proving the
decay in time of the solution derivatives in $L^2$-norm taken over
bounded sets. Such local energy decay is usually obtained in spatial
dimensions $d\geq 2$. Another recent paper \cite{Charao} contains
a time-decay result in the one-dimensional case under certain restrictions on $\alpha$, $\beta$ and localisation of the initial data $u_0$, $u_1$.

It is the purpose of the present work to improve such $L^2$-estimates
of the solution derivatives by showing that, in case of compactly
supported initial data, the decay of the solution in time is
exponential rather than algebraic. We also prove
that, in this case, an auxiliary condition on $\alpha$, $\beta$ used in \cite{Charao} can be avoided.
Moreover, we generalise the result to prove the decay not only of derivatives but also of the solution function itself. Namely, we conclude that the solution converges exponentially fast in time to a constant, which is explicitly computable in terms of $\alpha$, $\beta$ and $u_1$.

Let us denote $\mathbb{R}_+:=\left\{x\in\mathbb{R}: x\geq 0\right\}$
and start with a well-posedness result for
problem \eqref{eq:wave_1D}.

\begin{prop}\label{prop:WP}
Let $\alpha\in W^{1,\infty}\left(\mathbb{R}\right)$ and $\beta\in L^{\infty}\left(\mathbb{R}\right)$
with $\alpha\left(x\right)\geq\alpha_{\min}>0$ and $\beta\left(x\right)\geq\beta_{\min}>0$
for $x\in\mathbb{R}$. Furthermore,
let $u_{0}\in H^{1}\left(\mathbb{R}\right)$
and $u_{1}\in L^{2}\left(\mathbb{R}\right)$. Then, problem \eqref{eq:wave_1D} has
a unique solution $u\in C^{1}\left(\mathbb{R}_{+},L^{2}\left(\mathbb{R}\right)\right)\cap C\left(\mathbb{R}_{+},H^{1}\left(\mathbb{R}\right)\right)$, with the continuous dependence on the initial data, according to the following estimate
\begin{equation}\label{eq:WP_estim}
\left\Vert u\left(\cdot,t\right)\right\Vert _{H^{1}\left(\mathbb{R}\right)}+\left\Vert \partial_{t}u\left(\cdot,t\right)\right\Vert _{L^{2}\left(\mathbb{R}\right)}\leq C\left(\left\Vert u_{0}\right\Vert _{H^{1}\left(\mathbb{R}\right)}+\left\Vert u_{1}\right\Vert _{L^{2}\left(\mathbb{R}\right)}\right),\hspace{1em} 0\leq t\leq T,
\end{equation}
valid for any $T>0$, with some constant $C=C\left(T,\alpha,\beta\right)>0$.

\noindent Moreover, if $u_{0}\in H^{2}\left(\mathbb{R}\right)$ and $u_{1}\in H^{1}\left(\mathbb{R}\right)$,
then we have $u\in C^{2}\left(\mathbb{R}_{+},L^{2}\left(\mathbb{R}\right)\right)\cap C^{1}\left(\mathbb{R}_{+},H^{1}\left(\mathbb{R}\right)\right)\cap C\left(\mathbb{R}_{+},H^{2}\left(\mathbb{R}\right)\right)$.
\end{prop}

{For the long-time analysis}, we shall now make more restrictive assumptions.

\begin{assm}\label{assm:alph_bet}
Let $\alpha$, $\beta\in W^{1,\infty}\left(\mR\right)$
be real-valued functions such that $\alpha(x)\geq\alpha_{\min}$, $\beta(x)\geq\beta_{\min}$ for $x\in\mR$,  and $\alpha(x)\equiv\alpha_0$, $\beta(x)\equiv\beta_0$
for $x\in\mR\backslash\Omega_{in}$, with some  
open bounded interval 
$\Omega_{in}\subset\mR$ and constants $\alpha_{\min}$, $\beta_{\min}$, $\alpha_0$, $\beta_0>0$.
\end{assm}

\begin{assm}\label{assm:ICs}
Suppose that $u_{0}\in H^{1}\left(\mathbb{R}\right)$, $u_{1}\in
L^{2}\left(\mathbb{R}\right)$ are real-valued functions with the
supports $\supp u_{0}$, $\supp u_{1}\subset\Omega_0$ for some 
open bounded interval
$\Omega_0\subset\mathbb{R}$.
\end{assm}

Our main results are given by the following two theorems.
\begin{thm}\label{thm:1D_decay}
Assume that $\alpha$, $\beta$, $\Omega_{in}$ are as in Assumption \ref{assm:alph_bet}, and $u_0$, $u_1$, $\Omega_0$  are as in Assumption \ref{assm:ICs}.
Then, for any 
open bounded interval 
$\Omega\subset\mathbb{R}$, the solution of \eqref{eq:wave_1D} obeys the decay estimate
	\begin{equation}
	\left\Vert u\left(\cdot,t\right)-u_{\infty}\right\Vert _{H^{1}\left(\Omega\right)}+\left\Vert \partial_{t}u\left(\cdot,t\right)\right\Vert _{L^{2}\left(\Omega\right)}\leq Ce^{-\Lambda t},\hspace{1em}t\geq 0,\label{eq:u_exp_conv}
	\end{equation}
	for some constant $C=C\left(u_{0},u_{1},\alpha,\beta,\left|\Omega_0\right|,\left|\Omega_{in}\right|\right){>0}$,
with $\left|\Omega_0\right|$, $\left|\Omega_{in}\right|$ denoting the Lebesgue measure of the sets $\Omega_0$, $\Omega_{in}$, respectively, and 
	\begin{equation}
	u_{\infty}:=\frac{1}{2\left(\alpha_0\beta_0\right)^{1/2}}\int_{\Omega_0}u_{1}\left(x\right)\beta\left(x\right)dx,\hspace{1em}
	 {\Lambda:=\frac{\gamma_{0}}{2 t_0}e^{-2\gamma_{0}}\left(1-4\gamma_{0}e^{-4\gamma_{0}}\right)^{1/2}},   \label{eq:u_infty_1D}
	\end{equation}
	\begin{equation}
\gamma_0:=\frac{t_0}{4}\left\Vert \frac{\alpha^\prime}{\left(\alpha\beta\right)^{1/2}}+\left(\frac{\alpha}{\beta^3}\right)^{1/2}\beta^\prime\right\Vert_{L^\infty\left(\Omega_{in}\right)},\hspace{1em} t_0:=\int_{\Omega_{in}}\left(\frac{\beta\left(x\right)}{\alpha\left(x\right)}\right)^{1/2}dx.\label{eq:gam0_t0}
	\end{equation}
\end{thm}

We remark that the steady state $u_{\infty}$ given in
\eqref{eq:u_infty_1D} generalises the one resulting from the d'Alembert
formula for constant $\alpha$ and $\beta$, and for $u_0$, $u_1$ with compact
support:
$$
u_\infty = \frac12 \sqrt{\frac{\beta_0}{\alpha_0}}
\int_{\mathbb R} 
u_1(x)\, dx.
$$

Note that even though the decay \eqref{eq:u_exp_conv} is exponential, the decay rate $\Lambda$ 
might be be rather small in case of significantly varying coefficients $\alpha$, $\beta$ or large $\Omega_{in}$. In particular, the estimate for the decay rate given by the second equation of \eqref{eq:u_infty_1D} is exponentially small for large values of $\gamma_0$. 
Therefore, an alternative estimate is provided by the following
theorem, which employs a different idea and yields a better decay
rate, but in return requires an
assumption on the variation of~$\alpha$ and~$\beta$.

\begin{thm}\label{thm:better_rate}
Assume that $\alpha$, $\beta$, $\Omega_{in}$ are as in Assumption
\ref{assm:alph_bet}, $u_0$, $u_1$, $\Omega_0$  are as in Assumption
\ref{assm:ICs}, and $t_0$ is
  defined in \eqref{eq:gam0_t0}. 
Moreover, suppose 
that
\begin{equation}\label{eq:b0_cond}
b_0:=\left\Vert \alpha\right\Vert _{L^{\infty}\left(\Omega_{in}\right)}^{3/2}\left\Vert \beta\right\Vert _{L^{\infty}\left(\Omega_{in}\right)}^{1/2}\left\Vert \left(\frac{1}{\left(\alpha\beta\right)^{1/2}}\right)^{\prime}\right\Vert _{L^{2}\left(\Omega_{in}\right)}^{2}<\frac{1}{t_0}.
\end{equation}
Then, for the solution of \eqref{eq:wave_1D}, estimate \eqref{eq:u_exp_conv} holds true, 
and the decay rate is given by 
\begin{equation}\label{Lambda}
\Lambda=\frac{1}{2t_0}\left|\log\left(b_{0}t_0\right)\right|.
\end{equation}
\end{thm}

\begin{rem}
Theorems~\ref{thm:1D_decay} and~\ref{thm:better_rate} are valid
  also for complex-valued $u_{0}$, $u_{1}$ in Assumption~\ref{assm:ICs}.
Indeed, due to the linearity of problem \eqref{eq:wave_1D} and
the real-valuedness of $\alpha$, $\beta$, both theorems could be
applied separately to the real and imaginary parts of the initial data.
\end{rem}

Theorems~\ref{thm:1D_decay} and~\ref{thm:better_rate} are to be compared with the following recent result from~\cite{Charao}.

\begin{prop}{\cite[Thm. 1.2 (for $d=1$)]{Charao}}
\label{prop:Charao}
Assume that $\alpha\left(x\right)\equiv
\alpha_0$ and that $\beta$ satisfies Assumption \ref{assm:alph_bet}. Moreover, suppose that $\beta$ is such that
\begin{equation}\label{eq:eta_cond}
\eta:=\left\Vert \beta\right\Vert _{L^{\infty}\left(\Omega_{in}\right)}^{1/2}\left\Vert \left(\frac{1}{\beta^{1/2}}\right)^\prime\right\Vert _{L^{\infty}\left(\Omega_{in}\right)}{\left|\Omega_{in}\right|}<1.
\end{equation}
Then, for any bounded $\Omega\subset\mathbb{R}$, the solution of \eqref{eq:wave_1D} satisfies 
\begin{equation}\label{eq:Charao}
\left\Vert \partial_x u\left(\cdot,t\right)\right\Vert _{L^{2}\left(\Omega\right)}+\left\Vert \partial_{t}u\left(\cdot,t\right)\right\Vert _{L^{2}\left(\Omega\right)}\leq\frac{C}{t^{\left(1-\eta\right)/2}},\hspace{1em}t\geq t_1,
\end{equation}
for some constant $C>0$ and sufficiently large $t_1>0$.
\end{prop}

Note that, when the initial data are compactly supported, the
algebraic decay in \eqref{eq:Charao}  appears to be sub-optimal compared to both
Theorem~\ref{thm:1D_decay} and Theorem~\ref{thm:better_rate}, which
show the exponential decay of the solution derivatives. As {observed}
in \cite[Rem.~1.2]{Charao}, the smallness assumption
\eqref{eq:eta_cond} on $\eta$ measuring the relative perturbation of
the coefficients
is essential in the proof there,
but it is not needed in our proof of Theorem~\ref{thm:1D_decay}.

For the wave equation \eqref{eq:wave_1D} with \emph{constant}
coefficients and compactly supported initial data, the solution
converges on all bounded domains $\Omega$ to the constant $u_\infty$
in \emph{finite time}. However, for variable coefficients, convergence in finite time does not hold and exponential
decay is the generic scenario: as an example, consider
\eqref{eq:wave_1D} with $\alpha(x)\equiv \alpha_0$ in
$\mR$ and {$\beta$} Lipschitz with $\beta(x)=\beta_1$ on
a closed subinterval of some open set $\Omega_{in}$, and $\beta(x)=\beta_0$ in
$\mR\setminus\Omega_{in}$. A localised, travelling wave packet starting within $\Omega_{in}$ gets partly reflected and partly transmitted. Hence, its local norm $\left\Vert u\left(\cdot,t\right)-u_{\infty}\right\Vert _{H^{1}\left(\Omega_{in}\right)}$ decays by a constant factor at each of these reflections, giving rise to an exponential decay.
On the other hand, we note that, for non-compactly supported initial data, the decay of the derivatives of the solution in time is generally related to the decay of the initial data at infinity. This can be easily seen already for the constant-coefficient case: the solution furnished explicitly by the d'Alembert formula decays to some $u_\infty$ only algebraically fast if the assumed decay at infinity of $u_0$ or $u_1$ is algebraic. 

The following Sections~\ref{sec:1D_WP},~\ref{sec:1D_decay_proof},
  and~\ref{sec:better_rate_proof} are devoted to the proof of
  Proposition \ref{prop:WP}, Theorem~\ref{thm:1D_decay}, and Theorem~\ref{thm:better_rate},
  respectively. Some conclusions are drawn in Section~\ref{sec:conc}, and 
  technical lemmas used in the proof of Theorem~\ref{thm:1D_decay} are
  deferred to the Appendix.


\section{Proof of Proposition \ref{prop:WP}}\label{sec:1D_WP}
One standard approach {to prove} 
the well-posedness of problem \eqref{eq:wave_1D} is
to construct the solution $u$, for coefficients $\alpha$ and $\beta$
sufficiently regular, by a Galerkin approximation (e.g.\ see \cite[Sect. 7.2]{Evans}, or \cite[Par. 11.2]{Ren-Rog}). Here, we shall, however,
use semigroup theory tools, since they require lower regularity assumptions
on $\alpha$ and $\beta$. {Although} 
this approach is quite standard (e.g.
see \cite[Par. 2]{Reed},  \cite[Par. 3]{Strauss}, or
\cite[Sect. 12.3.2]{Ren-Rog}), we 
still outline
it, adapting to our particular setting. 

With the notation 
\[
\phi\left(x,t\right):=\left(u\left(x,t\right),v\left(x,t\right)\right)^{T},\hspace{1em}\hspace{1em}v\left(x,t\right):=\partial_{t}u\left(x,t\right),
\]
\[
A=A\left(x\right):=i\left(\begin{array}{cc}
0 & 1\\
\frac{1}{\beta\left(x\right)}\partial_{x}\left(\alpha\left(x\right){\partial_{x}\,\cdot}\,\right)-\frac{1}{\beta\left(x\right)} & 0
\end{array}\right),\hspace{1em}\hspace{1em}C=C\left(x\right):=\left(\begin{array}{cc}
0 & 0\\
\frac{1}{\beta\left(x\right)} & 0
\end{array}\right),
\]
problem \eqref{eq:wave_1D} is equivalent to the first order linear system
\begin{equation}\label{eq:1storder}
\begin{cases}
  \partial_{t}\phi\left(x,t\right)={\left(-iA \left(x\right)+C \left(x\right)\right)}\phi\left(x,t\right),
  & x\in\mathbb{R},\hspace{1em}t>0,\\
\phi\left(x,0\right)=\phi_{0}\left(x\right):=\left(u_{0}\left(x\right),u_{1}\left(x\right)\right)^{T}, & x\in\mathbb{R}.
\end{cases}
\end{equation}
We split the proof into 2 steps.

\subsection*{Step 1. Functional setting.}
We {set
$L^{2}_{\beta}\left(\mathbb{R}\right):=L^{2}\left(\mathbb{R};\beta\left(x\right)dx\right)$}. Due to the assumption $\beta$, $1/\beta\in
L^{\infty}\left(\mathbb{R}\right)$, {the space} ${L^{2}_{\beta}\left(\mathbb{R}\right)}$ 
actually consists of the same functions as~$L^{2}\left(\mathbb{R}\right)$.
We define the operator $B^{2}:=-\frac{1}{\beta\left(x\right)}\partial_{x}\left(\alpha\left(x\right){\partial_{x}\,\cdot \,}\right)+\frac{1}{\beta\left(x\right)}$
in ${L^{2}_{\beta}\left(\mathbb{R}\right)}$. 
Since $\alpha$, $\beta$ are assumed to be positive and bounded away from zero, the operator $B^2$ is strictly positive. 
We claim that this operator is self-adjoint with its (maximal) domain
being $\Dom B^{2}=H^{2}\left(\mathbb{R}\right)$. Indeed, the
symmetry of the operator $B^{2}$ is evident, and {$w\in
  H^2\left(\mathbb{R}\right)$ implies $B^2w\in L^{2}_{\beta}\left(\mathbb{R}\right)$.}
By \cite[Lem 2.3]{Teschl}, it
remains to verify its surjectivity on ${L^{2}_{\beta}\left(\mathbb{R}\right)}$ 
or, equivalently, on
$L^{2}\left(\mathbb{R}\right)$.
The latter follows
from the fact that, for $\alpha\in W^{1,\infty}\left(\mathbb{R}\right)$,
the operator $\widetilde{B}^{2}:=-\partial_{x}\left(\alpha\left(x\right){\partial_{x}\,\cdot \,}\right)+1$
from $H^{2}\left(\mathbb{R}\right)$ into $L^{2}\left(\mathbb{R}\right)$
is strictly positive, and hence invertible.

Next, we define in ${L^{2}_{\beta}\left(\mathbb{R}\right)}$ 
the positive square
root $B:=\sqrt{B^{2}}$, whose
(maximal) domain is $\Dom B=H^{1}\left(\mathbb{R}\right)$. In fact,
for any $w\in H^{1}\left(\mathbb{R}\right)$, we have
\begin{align*}
\left\Vert Bw\right\Vert _{{L^{2}_{\beta}\left(\mathbb{R}\right)}}^{2} & =\left\langle Bw,Bw\right\rangle _{{L^{2}_{\beta}\left(\mathbb{R}\right)}}=\int_{\mathbb{R}}\left(B^{2}w\right)\left(x\right)w\left(x\right)\beta\left(x\right)dx\\
 & =\int_{\mathbb{R}}\left[\alpha\left(x\right)\left(w^{\prime}\left(x\right)\right)^{2}+w^{2}\left(x\right)\right]dx,
\end{align*}
which, due to positivity and boundedness of $\alpha$, shows the
equivalence {between $\left\Vert B\,\cdot\,\right\Vert _{{L^{2}_{\beta}\left(\mathbb{R}\right)}}$ and
$\left\Vert \,\cdot\,\right\Vert _{H^{1}\left(\mathbb{R}\right)}$}.

With this preparation, we can introduce the Hilbert space $X:=H^{1}\left(\mathbb{R}\right)\times L^{2}\left(\mathbb{R}\right)$
equipped with the
{inner product $\langle \phi,\widetilde \phi\rangle_X:=
  \langle Bu,B\widetilde
  u\rangle_{L^{2}_{\beta}\left(\mathbb{R}\right)}+
  \langle v,\widetilde
  v\rangle_{L^{2}_{\beta}\left(\mathbb{R}\right)}$}.
Note that we have $\phi_{0}\in X$.

\subsection*{Step 2. Evolution semigroup.}

Observe that the matrix operator $A$ is symmetric in
$X$, since for any $\phi$, $\widetilde{\phi}\in H^{2}\left(\mathbb{R}\right)\times H^{1}\left(\mathbb{R}\right)$, we have 
\begin{align*}
\left\langle A\phi,\widetilde{\phi}\right\rangle _{X} & =i\left\langle Bv,B\widetilde{u}\right\rangle _{{L^{2}_{{\beta}}\left(\mathbb{R}\right)}}+i\left\langle -B^{2}u,\widetilde{v}\right\rangle _{{L^{2}_{{\beta}}\left(\mathbb{R}\right)}}\\
&=i\left\langle v,B^{2}\widetilde{u}\right\rangle_{{L^{2}_{{\beta}}\left(\mathbb{R}\right)}}-i\left\langle Bu,B\tilde{v}\right\rangle _{{L^{2}_{{\beta}}\left(\mathbb{R}\right)}}
 =\left\langle \phi,A\widetilde{\phi}\right\rangle _{X}.
\end{align*}
In fact, $A$ is self-adjoint in $X$ with (maximal) domain $\Dom A=H^{2}\left(\mathbb{R}\right)\times H^{1}\left(\mathbb{R}\right)$.
The latter statement follows from the maximality of $\Dom B^{2}=H^{2}\left(\mathbb{R}\right)$.

By Stone's theorem \cite[Ch.\ 1 Thm.\ 10.8]{Pazy}, the operator
$-iA$ generates a $C_{0}$-semigroup of unitary operators on $X$,
and the operator $C$ is a bounded perturbation. Therefore, using
results on the perturbation of semigroups \cite[Sect.\ 3.1 Thm.\ 1.1]{Pazy},
the operator $-iA+C${, which defines problem~\eqref{eq:1storder},} generates a $C_{0}$-semigroup of bounded
operators on $X$.
Consequently, we have the following two solution concepts {for
  problem~\eqref{eq:1storder}, and thus for problem~\eqref{eq:wave_1D}}, depending
on the regularity of $\phi_{0}=\left(u_{0},u_{1}\right)^{T}$ (see
\cite[p.\ 105]{Pazy} for this discussion):
\begin{itemize}
\item[ \emph{i)}] For $\phi_{0}\in X$,
problem~{\eqref{eq:1storder}} has a unique mild solution, i.e. $\phi\in C\left(\mathbb{R}_{+},X\right)$,
and hence
$u\in C\left(\mathbb{R}_{+},H^{1}\left(\mathbb{R}\right)\right)\cap C^{1}\left(\mathbb{R}_{+},L^{2}\left(\mathbb{R}\right)\right)$,
$v=\partial_{t}u\in C\left(\mathbb{R}_{+},L^{2}\left(\mathbb{R}\right)\right)$
\item[\emph{ii)}] For $\phi_{0}\in\text{Dom }A,$
  problem~{\eqref{eq:1storder}} has a unique classical
solution (in the semigroup sense), i.e. $\phi\in C^{1}\left(\mathbb{R}_{+},X\right)\cap C\left(\mathbb{R}_{+},\text{Dom }A\right)$,
and hence $u\in C\left(\mathbb{R}_{+},H^{2}\left(\mathbb{R}\right)\right)\cap C^{1}\left(\mathbb{R}_{+},H^{1}\left(\mathbb{R}\right)\right)$,
$v=\partial_{t}u\in C^{1}\left(\mathbb{R}_{+},L^{2}\left(\mathbb{R}\right)\right)$.
\end{itemize}  
Finally, estimate \eqref{eq:WP_estim} is an automatic consequence of the semigroup approach due to \cite[Sect.\ 1.2 Thm.\ 2.2]{Pazy}.
\qed


\section{Proof of Theorem \ref{thm:1D_decay}}\label{sec:1D_decay_proof}

Without loss of generality, let us take
        $\Omega=\Omega_{in}=\Omega_0=\left(0,x_0\right)$ for some
        $x_0>0$. This is possible since the decay in
        the region of interest $\Omega$ can be deduced by enlarging it (if necessary) to an interval
         containing both $\Omega_0$ and $\Omega_{in}${; moreover, as the
         problem is posed on the entire real line, we can choose the origin at the left end of the resulting interval}. 
	
{While the result of this theorem is stated for the mild solution 
\begin{equation}\label{eq:u_reg}
  u\in C^{1}\left(\mathbb{R}_{+},L^{2}\left(\mathbb{R}\right)\right)\cap C\left(\mathbb{R}_{+},H^{1}\left(\mathbb{R}\right)\right),
\end{equation}
its proof requires
the regularity of classical solutions established in
Proposition~\ref{prop:WP}. 
Hence we approximate the initial data $(u_0,\, u_1)^T\in
X{:=H^1\left(\mR\right)\times L^2\left(\mR\right)}$ by a sequence
{of functions} $(u_0^n,\, u_1^n)^T_{n\in\mN}\subset H^2\left(\mR\right)\times
H^1\left(\mR\right)$ with support in $\Omega_0$ 
and converging in $X$ to $(u_0,\, u_1)^T$. Due to Proposition \ref{prop:WP}, each approximate initial condition $(u_0^n,\, u_1^n)$ gives rise to a classical solution
\begin{equation}\label{eq:un_reg}
  u^n\in C^{2}\left(\mathbb{R}_{+},L^{2}\left(\mathbb{R}\right)\right)\cap C^{1}\left(\mathbb{R}_{+},H^{1}\left(\mathbb{R}\right)\right)\cap C\left(\mathbb{R}_{+},H^{2}\left(\mathbb{R}\right)\right).
\end{equation}
By the linearity of problem \eqref{eq:wave_1D} and bound \eqref{eq:WP_estim}}, these classical solutions converge uniformly to {the
  mild solution} $u$ on any finite time interval $[0,T]$:
\begin{equation}\label{un-conv}
  u^n \:\stackrel{n\to \infty}{\longrightarrow}\: u \quad \mbox{in } C^1([0,T],L^2(\mR))\cap C([0,T],H^1(\mR)).
\end{equation}

\noindent
We recall that the notions of classical and mild solutions are discussed in \cite[p.~105]{Pazy}.
\smallskip

The proof is performed in 3 steps. 
In Step~1,
        {for any fixed $n\in\mN$, we consider
          problem~\eqref{eq:wave_1D} for initial data $(u_0^n,\,
          u_1^n)\in H^2\left(\mR\right)\times
H^1\left(\mR\right)$ and classical solution $u^n$ as
in~\eqref{eq:un_reg}. W}e
        transform \eqref{eq:wave_1D} into an auxiliary
        problem, for which we construct a weighted, local energy
        functional that would {admit} 
         an exponential decay under
        certain conditions on weight functions.
        In Step~2,
        we construct these weight functions such that the rate of decay of the energy functional can be estimated explicitly.
	In Step~3,
        we use the local energy decay {and the convergence in~\eqref{un-conv}}
        to prove that the {mild} solution {$u$} of~\eqref{eq:wave_1D}, for large
	times, converges
        to a constant uniformly on $\left[0,x_{0}\right]$. Moreover,
	we identify this constant as $u_{\infty}$ defined in \eqref{eq:u_infty_1D}.
	
\subsection*{Step 1. Auxiliary local energy functional.}
{{As anticipated,} here and in Step~2, we 
  establish the exponential-in-time decay of any classical solution $u^n$, i.e.\ for fixed index $n\in\mN$.}

Let us perform the
change of variable 
	\begin{equation}
	x\mapsto y\left(x\right):=\dfrac{1}{t_0}\int_{0}^{x}\sqrt{\beta\left(\xi\right)/\alpha\left(\xi\right)}d\xi ,\hspace{1em}x\left(s\right)=y^{-1}\left(s\right),\label{eq:y_def}
	\end{equation}
	\begin{equation}
	t\mapsto \tau\left(t\right):=t/t_0,\hspace{1em}t\left(\tau\right)=t_0\tau,\hspace{1em} t_0:=\int_{0}^{x_0}\sqrt{\beta\left(\xi\right)/\alpha\left(\xi\right)}d\xi,\label{eq:tau_def}
	\end{equation}
and denote $v\left(y,\tau\right):={u^n}\left(x\left(y\right),t_0\tau\right)$, $v_{0}\left(y\right):={u^n_{0}}\left(x\left(y\right)\right)$,
$v_{1}\left(y\right):=t_0\, {u^n_{1}}\left(x\left(y\right)\right)$,
\begin{equation}\label{eq:rho_gam_def}
\rho\left(y\right):=\frac{1}{\sqrt{\alpha\left(x\left(y\right)\right)\beta\left(x\left(y\right)\right)}},\hspace{1em}
\gamma\left(y\right):=\frac{\rho^{\prime}\left(y\right)}{\rho\left(y\right)},\hspace{1em}\rho_0:=\rho\left(0\right)=\rho\left(1\right)=\frac{1}{\sqrt{\alpha_0\beta_0}}.
\end{equation}

Due to Assumption \ref{assm:alph_bet}, the map $x\mapsto y\left(x\right)$ is $W^{2,\infty}\left(\mathbb{R}\right)$, and hence $v$ inherits the regularity of {$u^n$} given by \eqref{eq:un_reg}:
\begin{equation}\label{eq:v_reg}
v\in C^{2}\left(\mathbb{R}_{+},L^{2}\left(\mathbb{R}\right)\right)\cap C^{1}\left(\mathbb{R}_{+},H^{1}\left(\mathbb{R}\right)\right)\cap C\left(\mathbb{R}_{+},H^{2}\left(\mathbb{R}\right)\right).
\end{equation}
We thus arrive at a problem equivalent to \eqref{eq:wave_1D}:
	\begin{equation}
	\begin{cases}
	\partial_{\tau}^{2}v\left(y,\tau\right)-\partial_{y}^{2}v\left(y,\tau\right)+\gamma\left(y\right)\partial_{y}v\left(y,\tau\right)=0, & y\in\mathbb{R},\hspace{1em}\tau>0,\\
	v\left(y,0\right)=v_{0}\left(y\right),\hspace{1em}\partial_{\tau}v\left(y,0\right)=v_{1}\left(y\right), & y\in\mathbb{R},
	\end{cases}\label{eq:v_pbm}
	\end{equation}
	where $\gamma\left(y\right)\equiv0$ for $y\in\mathbb{R}\backslash\left(0,1\right)$
	and $\text{supp }v_{0}$, $\text{supp
        }v_{1}\subseteq(0,1)$.
	
Motivated by the form of the global energy introduced in~\cite{Lewis} for the wave equation in~\eqref{eq:v_pbm},
we consider the local energy functional
	\begin{align}
	\mathcal{E}_{loc}\left(\tau\right):= & \frac{1}{2}\int_{0}^{1}\left[\left(\partial_{\tau}v\left(y,\tau\right)\right)^{2}+\left(\partial_{y}v\left(y,\tau\right)\right)^{2}\right]\frac{dy}{\rho\left(y\right)}\nonumber \\
	= & \frac{1}{4}\int_{0}^1\left[\left(\partial_{\tau}v\left(y,\tau\right)+\partial_{y}v\left(y,\tau\right)\right)^{2}+\left(\partial_{\tau}v\left(y,\tau\right)-\partial_{y}v\left(y,\tau\right)\right)^{2}\right]\frac{dy}{\rho\left(y\right)}>0.\label{eq:E_loc}
	\end{align}
This quantity is decaying since, by direct calculations employing
integration by parts, using \eqref{eq:rho_gam_def} and the wave
equation in \eqref{eq:v_pbm}, we have
\begin{align}\label{eq:E_aux_dt}
\mathcal{E}^\prime_{loc}\left(\tau\right)=&\int_{0}^{1}\left[
\partial_{\tau}v\left(y,\tau\right)\partial^2_{\tau}v\left(y,\tau\right)
+ \partial_{y}v\left(y,\tau\right){\partialytau} v\left(y,\tau\right)
                                            \right]\frac{dy}{\rho\left(y\right)}
  \nonumber \\
  {=}& {\int_{0}^{1}\left[\partial_{\tau}v\left(y,\tau\right)
   \left(\partial^2_{\tau}v\left(y,\tau\right)
   -\partial_{y}^{2}v\left(y,\tau\right)+\frac{\rho^{\prime}\left(y\right)}{\rho\left(y\right)}\partial_{y}v\left(y,\tau\right)\right)                                           
\right]\frac{dy}{\rho\left(y\right)}}\\
&{+\left[\partial_y v\left(1,\tau\right) \partial_\tau
        v\left(1,\tau\right)-\partial_y v\left(0,\tau\right)
        \partial_\tau v\left(0,\tau\right)\right]\frac{1}{\rho_0}} \nonumber\\
=&\left[\partial_y v\left(1,\tau\right) \partial_\tau v\left(1,\tau\right)-\partial_y v\left(0,\tau\right) \partial_\tau v\left(0,\tau\right)\right]\frac{1}{\rho_0}\le0. \nonumber
\end{align}
{Note that we have used here the regularity of a classical
  solution.} {Moreover}, in order to deduce the sign in the last line, we have used the
exact outflow
boundary conditions 
(well-defined for any {$\tau \geq 0$} due to \eqref{eq:v_reg}):
\begin{equation}\label{eq:BCs_y}
\left.\left(\partial_{y}-\partial_{\tau}\right)v\left(y,\tau\right)\right|_{y=0}=0=\left.\left(\partial_{y}+\partial_{\tau}\right)v\left(y,\tau\right)\right|_{y=1},
\end{equation}
which are due to the support properties of the functions $v_0$, $v_1$ and $\gamma$.
 
We note, in passing, that the global energy, defined in the same way as \eqref{eq:E_loc} but with the integration range replaced by $\mathbb{R}$, is conserved.
	
Even though the conventional local energy functional \eqref{eq:E_loc} can (and, in the proof of Theorem \ref{thm:better_rate}, will) be used to show the exponential decay of the solution derivatives under some quantitative restriction on the coefficients $\alpha$ and $\beta$, here, we proceed with an alternative strategy that does not require such a restriction.
To this effect, we consider now a modified local energy functional in the spirit of \cite[Sec.~3]{Arnold}:
	\begin{equation}
	E_{loc}\left(\tau\right):=\frac{1}{2}\int_{0}^{1}\left[\phi_{1}\left(y\right)\left(\partial_{\tau}v\left(y,\tau\right)+\partial_{y}v\left(y,\tau\right)\right)^{2}+\phi_{2}\left(y\right)\left(\partial_{\tau}v\left(y,\tau\right)-\partial_{y}v\left(y,\tau\right)\right)^{2}\right]\frac{dy}{\rho\left(y\right)},\label{eq:E_loc_mod}
	\end{equation}
	where $\phi_{1}\left(y\right)$, $\phi_{2}\left(y\right)$ are some
	strictly positive weight functions on $\left[0,1\right]$, which are yet
	to be chosen.
	
Similarly to \eqref{eq:E_aux_dt}, we differentiate under the integral sign to obtain (suppressing the arguments of all functions for the sake of brevity)
\begin{align}\label{eq:E_dt_estim_prelim}
E_{loc}^{\prime}\left(\tau\right)=&\int_{0}^{1}\left[\phi_{1}\left(\partial_{\tau}v+\partial_{y}v\right)\left(\partial_{\tau}^{2}v+{\partialytau}v\right)+\phi_{2}\left(\partial_{\tau}v-\partial_{y}v\right)\left(\partial_{\tau}^{2}v-{\partialytau}v\right)\right]\frac{dy}{\rho}\\
=&\int_{0}^{1}\left[\phi_{1}\left(\partial_{\tau}v+\partial_{y}v\right)\left(\partial_{y}^{2}v-\gamma\partial_{y}v+{\partialytau}v\right)+\phi_{2}\left(\partial_{\tau}v-\partial_{y}v\right)\left(\partial_{y}^{2}v-\gamma\partial_{y}v-{\partialytau}v\right)\right]\frac{dy}{\rho}\nonumber\\
=&{\frac12}\int_{0}^{1}\left[\phi_{1}\partial_{y}\frac{\left(\partial_{\tau}v+\partial_{y}v\right)^{2}}{\rho}-\phi_{2}\partial_{y}\frac{\left(\partial_{\tau}v-\partial_{y}v\right)^{2}}{\rho}+\left(\phi_{1}-\phi_{2}\right)\gamma\frac{\left(\partial_{\tau}v\right)^{2}-\left(\partial_{y}v\right)^{2}}{\rho}\right]dy,\nonumber
\end{align}
where we used the wave equation in~\eqref{eq:v_pbm} to eliminate
second-order time derivatives, and the definition of $\gamma$ in~\eqref{eq:rho_gam_def}.
Integrating \eqref{eq:E_dt_estim_prelim} by parts and employing boundary conditions \eqref{eq:BCs_y}, 
	we obtain
	\begin{align*}
	E_{loc}^{\prime}\left(\tau\right)= & \frac{1}{2}\int_{0}^{1}\left[-\phi_{1}^{\prime}\left(y\right)\left(\partial_{\tau}v\left(y,\tau\right)+\partial_{y}v\left(y,\tau\right)\right)^{2}+\phi_{2}^{\prime}\left(y\right)\left(\partial_{\tau}v\left(y,\tau\right)-\partial_{y}v\left(y,\tau\right)\right)^{2}\right.\\
	& \left.+\left(\phi_{1}\left(y\right)-\phi_{2}\left(y\right)\right)\gamma\left(y\right)\left(\left(\partial_{\tau}v\left(y,\tau\right)\right)^{2}-\left(\partial_{y}v\left(y,\tau\right)\right)^{2}\right)\right]\frac{dy}{\rho\left(y\right)}\\
	&-\left[\phi_{2}\left(1\right)\left(\partial_{y}v\left(1,\tau\right)\right)^{2}+\phi_{1}\left(0\right)\left(\partial_{y}v\left(0,\tau\right)\right)^{2}\right]\frac{2}{\rho_0}.
	\end{align*}
	We estimate
	\begin{align*}
	\left(\phi_{1}-\phi_{2}\right)\gamma\left(\left(\partial_{\tau}v\right)^{2}-\left(\partial_{y}v\right)^{2}\right) & =\left(\phi_{1}-\phi_{2}\right)\gamma\left(\partial_{\tau}v+\partial_{y}v\right)\left(\partial_{\tau}v-\partial_{y}v\right)\\
	& \leq\frac{1}{2}\left|\phi_{1}-\phi_{2}\right|\left|\gamma\right|\left(\left(\partial_{\tau}v+\partial_{y}v\right)^{2}+\left(\partial_{\tau}v-\partial_{y}v\right)^{2}\right),
	\end{align*}
	and therefore, since $\rho>0$, we arrive at
	\begin{align}
	E_{loc}^{\prime}\left(\tau\right)\leq & -\frac{1}{2}\int_{0}^{1}\left[\left(\phi_{1}^{\prime}\left(y\right)-\frac{\left|\gamma\left(y\right)\right|}{2}\left|\phi_{1}\left(y\right)-\phi_{2}\left(y\right)\right|\right)\left(\partial_{\tau}v\left(y,\tau\right)+\partial_{y}v\left(y,\tau\right)\right)^{2}\right.\nonumber \\
	& \left.+\left(-\phi_{2}^{\prime}\left(y\right)-\frac{\left|\gamma\left(y\right)\right|}{2}\left|\phi_{1}\left(y\right)-\phi_{2}\left(y\right)\right|\right)\left(\partial_{\tau}v\left(y,\tau\right)-\partial_{y}v\left(y,\tau\right)\right)^{2}\right]\frac{dy}{\rho\left(y\right)}.\label{eq:E_loc_dt}
	\end{align}

We aim to obtain
an estimate of the form $E_{loc}^{\prime}\left(\tau\right)\leq -\lambda_0 E_{loc}\left(\tau\right)$ with some $\lambda_0>0$, 
and explore the possibility of choosing
the weight functions $\phi_1\left(y\right)$, $\phi_2\left(y\right)$
that would lead to
such an estimate. In \cite{Arnold} it is shown in
a more general setting that such weights do exist theoretically. Here,
we present an explicit
construction of $\phi_1$, $\phi_2$, which will permit us to obtain an explicit estimate of the decay rate of the solution. 
	

\subsection*{Step 2. Construction of the weights $\phi_1$, $\phi_2$.}

Motivated by \eqref{eq:E_loc_dt}, we shall now construct functions $\phi_{1}\left(y\right)$, $\phi_{2}\left(y\right)>0$ in \eqref{eq:E_loc_mod}
	that satisfy, on $\left[0,1\right]$, the following differential inequalities
	\begin{equation}
	\begin{cases}
	\phi_{1}^{\prime}\left(y\right)-\frac{\left|\gamma\left(y\right)\right|}{2}\left|\phi_{1}\left(y\right)-\phi_{2}\left(y\right)\right|\geq\lambda\phi_{1}\left(y\right),\hspace{1em}y\in\left(0,1\right),\\
	-\phi_{2}^{\prime}\left(y\right)-\frac{\left|\gamma\left(y\right)\right|}{2}\left|\phi_{1}\left(y\right)-\phi_{2}\left(y\right)\right|\geq\lambda\phi_{2}\left(y\right),\hspace{1em}y\in\left(0,1\right),\\
	\phi_{1}\left(0\right)=\phi_{2}\left(0\right)=1,
	\end{cases}\label{eq:phi_syst}
	\end{equation}
	for some constant $\lambda>0$ to be chosen. To this end,
	we consider functions $\varphi_{1}\left(y\right)$, $\varphi_{2}\left(y\right)>0$ solving the linear ODE system
	\begin{equation}
	\begin{cases}
	\varphi_{1}^{\prime}\left(y\right)-\gamma_0\left(\varphi_{1}\left(y\right)-\varphi_{2}\left(y\right)\right)=\lambda\varphi_{1}\left(y\right),\hspace{1em}y\in\left(0,1\right),\\
	-\varphi_{2}^{\prime}\left(y\right)-\gamma_0\left(\varphi_{1}\left(y\right)-\varphi_{2}\left(y\right)\right)=\lambda\varphi_{2}\left(y\right),\hspace{1em}y\in\left(0,1\right),\\
	\varphi_{1}\left(0\right)=\varphi_{2}\left(0\right)=1,
	\end{cases}\label{eq:phi_syst_lin}
	\end{equation}
where $\gamma_{0}:=\frac{1}{2}\left\Vert \gamma\right\Vert _{L^{\infty}\left(0,1\right)}$.

It is easy to see that any solution of \eqref{eq:phi_syst_lin} automatically satisfies \eqref{eq:phi_syst}, provided that $\varphi_{1}\left(y\right)$, $\varphi_{2}\left(y\right)>0$ for all $y\in\left[0,1\right]$. Indeed, it suffices to verify that
\begin{equation}\label{eq:phis_diff}
0\leq\varphi_{1}\left(y\right)-\varphi_{2}\left(y\right)=\left|\varphi_{1}\left(y\right)-\varphi_{2}\left(y\right)\right|, \hspace{1em} y\in\left[0,1\right].
\end{equation}	
Summing both ODEs of \eqref{eq:phi_syst_lin}, multiplying by
  $e^{-2\gamma_0y}$, and integrating on $(0,y)$,
we obtain
\[
\varphi_1\left(y\right)-\varphi_2\left(y\right)=\lambda\int_0^y e^{2\gamma_0\left(y-\xi\right)}\left[\varphi_1\left(\xi\right)+\varphi_2\left(\xi\right)\right]d\xi, \hspace{1em} y\in\left[0,1\right],
\]	
and hence positivity of the integrand directly implies \eqref{eq:phis_diff}.
 
Note that it is possible to consider more general versions of \eqref{eq:phi_syst} (and hence \eqref{eq:phi_syst_lin}), with equations involving different parameters $\lambda_1$, $\lambda_2$ on the right-hand sides, and initial conditions given by a positive constant different from $1$. However, it can be shown that no advantage could be gained from such generalisations. 	
	
We shall now estimate the largest possible value $\lambda>0$ (or, more precisely, its supremum) such that $\varphi_1$, $\varphi_2>0$ in $\left[0,1\right]$.
	According to the first equation in \eqref{eq:phi_syst_lin} and \eqref{eq:phis_diff}, $\varphi_{1}\left(y\right)$
	is a monotonically increasing function of $y$ for $y\in\left(0,1\right)$ and thus $\varphi_1\left(0\right)=1$ implies $\varphi_1\left(y\right)>1$ for $y\in\left(0,1\right]$. Similarly, it follows from the second equation of \eqref{eq:phi_syst_lin} that  $\varphi_{2}\left(y\right)$ decreases monotonically in $\left[0,1\right]$ starting from the value $\varphi_{2}\left(0\right)=1$.
	Therefore, to guarantee the positivity of $\varphi_{2}\left(y\right)$ in $\left[0,1\right]$,
	it suffices to require that $\varphi_2\left(1\right)>0$. 
We can write the solution of \eqref{eq:phi_syst_lin}
explicitly:
\begin{equation}\label{eq:phi1_sol}
\varphi_{1}\left(y\right)=\frac{e^{\gamma_{0}y}}{\sqrt{\lambda^{2}+\gamma_{0}^{2}}}\left[\sqrt{\lambda^{2}+\gamma_{0}^{2}}\cosh\left(\sqrt{\lambda^{2}+\gamma_{0}^{2}}y\right)+\left(\lambda-\gamma_{0}\right)\sinh\left(\sqrt{\lambda^{2}+\gamma_{0}^{2}}y\right)\right],
\end{equation}
\begin{equation}\label{eq:phi2_sol}
\varphi_{2}\left(y\right)=\frac{e^{\gamma_{0}y}}{\sqrt{\lambda^{2}+\gamma_{0}^{2}}}\left[\sqrt{\lambda^{2}+\gamma_{0}^{2}}\cosh\left(\sqrt{\lambda^{2}+\gamma_{0}^{2}}y\right)-\left(\lambda+\gamma_{0}\right)\sinh\left(\sqrt{\lambda^{2}+\gamma_{0}^{2}}y\right)\right],
\end{equation}

Then, the above requirement $\varphi_2\left(1\right)>0$ is equivalent to
\begin{equation}\label{eq:tanh_ineq1}
\tanh\left({\sqrt{\gamma_{0}^{2}+\lambda^{2}}}\right)<\frac{\sqrt{\gamma_0^2+\lambda^2}}{\gamma_0+\lambda}.
\end{equation}
To determine the interval of admissible values of $\lambda$
satisfying~\eqref{eq:tanh_ineq1}, we consider
the corresponding equality
\begin{equation}\label{eq:tanh_eq}
\tanh\left({\sqrt{\gamma_{0}^{2}+\lambda^{2}_{*}}}\right)=\frac{\sqrt{\gamma_0^2+\lambda^2_{*}}}{\gamma_0+\lambda_{*}},
\end{equation}
which implicitly defines a function $\lambda_{*}(\gamma_0)$ {on} $\mR_+$. The proof of the unique solvability of~\eqref{eq:tanh_eq} for
$\lambda_*$, as well as the fact that~\eqref{eq:tanh_ineq1} then holds
for all {$\lambda\in(0,\lambda_*)$}, is deferred to
Lemma~\ref{lem:tanh_eq} in the Appendix {(inequality~\eqref{eq:tanh_ineq1} is trivially satisfied for $\lambda=0$)}.
The function $\lambda_{*}(\gamma_0)$ is illustrated in Figure \ref{fig:lmb_vs_gam}. 
We see that $\lambda_{*}(\gamma_0)$ is monotonically decreasing with $\lambda_*\to0$ as $\gamma_0\to\infty$ (this can also be verified analytically by implicit differentiation of \eqref{eq:tanh_eq}). 

While the parameter choice $\lambda=\lambda_*$ is not admissible
for the construction of the weight functions (since
$\varphi_2\left(1\right)=0$ if $\lambda=\lambda_*$), any value
$\lambda\in(0,\lambda_{*})$ is admissible, and therefore
  guarantees positivity of $\varphi_1$, $\varphi_2$. 

\begin{figure}
\includegraphics[scale=0.8]{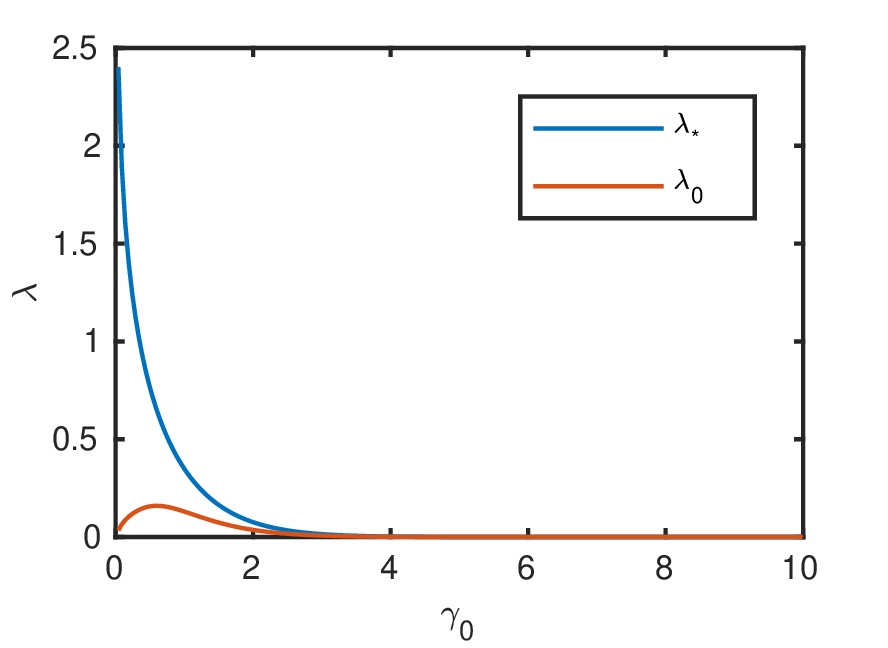}
\caption{Blue: function $\lambda_*(\gamma_0)$, implicitly defined by \eqref{eq:tanh_eq}. Red: function $\lambda_0(\gamma_0)$, defined in~\eqref{eq:lmb0_simpl1}.}\label{fig:lmb_vs_gam}
\end{figure}

Since the function $\lambda_{*}(\gamma_0)$ is not defined explicitly,
we give an example of an explicit lower bound that eventually yields the explicit (but suboptimal) decay rate stated in \eqref{eq:u_infty_1D}. The function 
\begin{equation}
{\lambda_0(\gamma_0):= \gamma_{0}e^{-2\gamma_{0}}\sqrt{1-4\gamma_{0}e^{-4\gamma_{0}}}\label{eq:lmb0_simpl1}}
\end{equation}
satisfies \eqref{eq:tanh_ineq1} and
$\lambda_0(\gamma_0)<\lambda_*(\gamma_0)$ for any $\gamma_0>0$,
and it is depicted in Figure \ref{fig:lmb_vs_gam}. The proof of
  the inequality $\lambda_0(\gamma_0)<\lambda_*(\gamma_0)$
is given in Lemma~\ref{lem:A2} in the Appendix. 

Therefore, we can take $\lambda=\lambda_0:=\lambda_0(\gamma_0)$, which satisfies inequality \eqref{eq:tanh_ineq1}. Hence the weights $\phi_1\left(y\right):=\varphi_1\left(y\right)>0$, $\phi_2\left(y\right):=\varphi_2\left(y\right)>0$ satisfy \eqref{eq:phi_syst}.
Consequently, we obtain from \eqref{eq:E_loc_dt} and~\eqref{eq:E_loc_mod}
\begin{equation}\label{eq:E_loc_dec}
E_{loc}^{\prime}\left(\tau\right)\leq-\lambda_0 E_{loc}\left(\tau\right),\hspace{1em}\tau>0\hspace{1em}\Longrightarrow\hspace{1em}E_{loc}\left(\tau\right)\le E_{loc}\left(0\right)e^{-\lambda_0 \tau},\hspace{1em}\tau\geq 0.
\end{equation}

	
\subsection*{Step 3. Convergence of the solution and identification of the constant~$u_{\infty}$.}
	
	The local energy functional \eqref{eq:E_loc_mod} can be rewritten in the original variables $\left(x,t\right)$, 
{and we {emphasise} here that it is a functional of the approximate classical solution $u^n$:}
	\begin{align*}
         {\widetilde{E}_{loc}[u^n]}\left(t\right):=E_{loc}\left(\frac{t}{t_0}\right)=&\frac{t_0}{2}\int_{0}^{x_{0}}\left[{\widetilde{\varphi}_1}
\left(x\right)\left(\partial_{t}u^n\left(x,t\right)+\sqrt{\frac{\alpha\left(x\right)}{\beta\left(x\right)}}\partial_{x}u^n\left(x,t\right)\right)^{2}\right.\\
 &\left.+{\widetilde{\varphi}_2}
 \left(x\right)\left(\partial_{t}u^n\left(x,t\right)-\sqrt{\frac{\alpha\left(x\right)}{\beta\left(x\right)}}\partial_{x}u^n\left(x,t\right)\right)^{2}\right]\beta\left(x\right)dx,
	\end{align*}
	where ${\widetilde{\varphi}_1}
        \left(x\right):=\varphi_{1}\left(y\left(x\right)\right)$,
	${\widetilde{\varphi}_1}
        \left(x\right):=\varphi_{2}\left(y\left(x\right)\right)$
	are given by \eqref{eq:phi1_sol}--\eqref{eq:phi2_sol}, with
        $\lambda=\lambda_0$ and the mapping $x\mapsto y\left(x\right)$
        given by \eqref{eq:y_def}. Then, from~\eqref{eq:E_loc_dec}, we obtain
$$
   {\widetilde{E}_{loc}[u^n]}\left(t\right)\le 
    {\widetilde{E}_{loc}[u^n]}\left(0\right)e^{-\frac{\lambda_0}{t_0} t},\hspace{1em}t\geq 0,
$$
{where the decay rate $\lambda_0/t_0$ is independent of the index $n$. Due to the convergence~\eqref{un-conv}{,} we can pass to the limit $n\to\infty$ and obtain the same inequality for the mild solution $u$:
\begin{equation}\label{ineq:tE-decay}
  \widetilde{E}_{loc}[u]\left(t\right)\le 
    \widetilde{E}_{loc}[u]\left(0\right)e^{-\frac{\lambda_0}{t_0} t},\hspace{1em}t\geq 0{,}
\end{equation}
and, due to the regularity of $u$ in~\eqref{eq:u_reg}, we have $\widetilde{E}_{loc}[u]\in C\left(\mathbb{R}_+\right)$.}

Setting 
\[
C_{0}:={\min}\left\{ \underset{y\in\left(0,1\right)}{\inf}\varphi_{1}\left(y\right),\underset{y\in\left(0,1\right)}{\inf}\varphi_{2}\left(y\right)\right\}={\min}\left\{\varphi_1\left(0\right),\varphi_2\left(1\right)\right\}=\varphi_2\left(1\right)>0,
\]
we have
	\begin{equation*}
	{\widetilde{E}_{loc}[u]}\left(t\right)\geq C_{0}t_0\int_{0}^{x_{0}}\left[\beta\left(x\right)\left(\partial_{t}u\left(x,t\right)\right)^{2}+\alpha\left(x\right)\left(\partial_{x}u\left(x,t\right)\right)^{2}\right]dx,
	\end{equation*}
	and thus deduce that
	\begin{equation}
	\left\Vert \partial_{t}u\left(\cdot,t\right)\right\Vert _{L^{2}\left(0,x_{0}\right)}^{2}\leq\frac{\left\Vert 1/\beta\right\Vert _{L^{\infty}\left(0,x_{0}\right)}}{C_{0}t_0} 
	 {\widetilde{E}_{loc}[u]}\left(t\right),\hspace{1em}\left\Vert \partial_{x}u\left(\cdot,t\right)\right\Vert _{L^{2}\left(0,x_{0}\right)}^{2}\leq\frac{\left\Vert 1/\alpha\right\Vert _{L^{\infty}\left(0,x_{0}\right)}}{C_{0}t_0} 
	  {\widetilde{E}_{loc}[u]}\left(t\right).\label{eq:u_x_t_bnds}
	\end{equation}
This{, together with~\eqref{ineq:tE-decay},} yields the exponential
decay of $\left\Vert \partial_{t}u\right\Vert
_{L^{2}\left(0,x_{0}\right)}$ and $\left\Vert \partial_{x}u\right\Vert
_{L^{2}\left(0,x_{0}\right)}$ claimed in \eqref{eq:u_exp_conv}. In
  order to complete the proof, it remains to identify $u_\infty$,
{and prove that, for all $t\ge 0$, $u(\cdot,t)$ converges in
  $H^1(0,x_0)$
to $u_\infty$.}

Using the regularity of $u$ given by \eqref{eq:u_reg},	
	we have, for $x\in\left(0,x_{0}\right)$, $t>0$,
	\[
	u\left(x,t\right)=u\left(0,t\right)+\int_{0}^{x}\partial_{x}u\left(s,t\right)ds=u\left(x_{0},t\right)-\int_{x}^{x_{0}}\partial_{x}u\left(s,t\right)ds,
      \]
      which implies
	\begin{equation}
          \left\Vert u\left(\cdot,t\right)-u\left(0,t\right)\right\Vert _{L^{2}\left(0,x_{0}\right)}^{2},\,\left\Vert u\left(\cdot,t\right)-u\left(x_{0},t\right)\right\Vert _{L^{2}\left(0,x_{0}\right)}^{2}\leq x_0^2\left\Vert \partial_{x}u\left(\cdot,t\right)\right\Vert _{L^{2}\left(0,x_{0}\right)}^{2}.\label{eq:u_bvp_bnds}
	\end{equation}
	
	Let us now integrate the wave equation in \eqref{eq:wave_1D} over
	the rectangle $\left(0,x_{0}\right)\times\left(0,t\right)$, $t>0$.
	Using the outflow boundary conditions \eqref{eq:u_reg}
	\[
	\left.\left(\partial_{x}-\sqrt{\frac{\beta_0}{\alpha_0}}\partial_{t}\right)u\left(x,t\right)\right|_{x=0}=0,\hspace{1em} \left.\left(\partial_{x}+\sqrt{\frac{\beta_0}{\alpha_0}}\partial_{t}\right)u\left(x,t\right)\right|_{x=x_{0}}=0,
	\]
	and the compact support of $u_{0}$, $u_{1}$, we obtain
	\begin{equation}
	\int_{0}^{x_{0}}\beta\left(s\right)\partial_{t}u\left(s,t\right)ds-\int_{0}^{x_{0}}\beta\left(s\right)u_{1}\left(s\right)ds=-\sqrt{\alpha_0\beta_0}\left(u\left(0,t\right)+u\left(x_{0},t\right)\right).\label{eq:u_IBVP_ident_prelim}
	\end{equation}
{Due to the weak regularity of the mild solution $u$, 
{the same} identity actually first needs to be derived for the classical
  solutions $u^n$ from \eqref{eq:un_reg}, and then{~\eqref{eq:u_IBVP_ident_prelim} follows by} passing to the
  limit} {$n\to\infty$.}
{From~}\eqref{eq:u_IBVP_ident_prelim}, we {deduce}
	\begin{align}
	u\left(x,t\right)-\frac{1}{2\sqrt{\alpha_0\beta_0}}\int_{0}^{x_{0}}\beta\left(s\right)u_{1}\left(s\right)ds=&\frac{1}{2}\left[u\left(x,t\right)-u\left(0,t\right)\right]+\frac{1}{2}\left[u\left(x,t\right)-u\left(x_{0},t\right)\right]\nonumber\\
	&-\frac{1}{2\sqrt{\alpha_0\beta_0}}\int_{0}^{x_{0}}\beta\left(s\right)\partial_{t}u\left(s,t\right)ds\label{eq:u_IBVP_ident}
	\end{align}
{for all $x\in\left(0,x_{0}\right)$ and all $t\ge0$.}
For the first two terms on the right-hand side
  of~\eqref{eq:u_IBVP_ident}, we have
	\[
\left\Vert u\left(x,t\right)-u\left(0,t\right)\right\Vert_{H^{1}\left(0,x_0\right)}^2=\left\Vert u\left(x,t\right)-u\left(0,t\right)\right\Vert_{L^{2}\left(0,x_0\right)}^2+\left\Vert \partial_{x}u\left(x,t\right)\right\Vert_{L^{2}\left(0,x_0\right)}^2,
	\]
	\[
\left\Vert u\left(x,t\right)-u\left(x_0,t\right)\right\Vert_{H^{1}\left(0,x_0\right)}^2=\left\Vert u\left(x,t\right)-u\left(x_0,t\right)\right\Vert_{L^{2}\left(0,x_0\right)}^2+\left\Vert \partial_{x}u\left(x,t\right)\right\Vert_{L^{2}\left(0,x_0\right)}^2.
	\] 
Moreover,
since the integral term on the right-hand side of
\eqref{eq:u_IBVP_ident} is independent of~$x$, 
we also have
{
  \[
    \begin{split}
    \left\Vert
 \frac{1}{2\sqrt{\alpha_0\beta_0}}\int_{0}^{x_{0}}\beta\left(s\right)\partial_{t}u\left(s,t\right)ds     
\right\Vert_{H^{1}\left(0,x_0\right)}
&=    \left\Vert
 \frac{1}{2\sqrt{\alpha_0\beta_0}}\int_{0}^{x_{0}}\beta\left(s\right)\partial_{t}u\left(s,t\right)ds     
\right\Vert_{L^{2}\left(0,x_0\right)}\\
&\le \frac{1}{2}\sqrt{\frac{x_0}{\alpha_0\beta_0}}
\left\Vert\beta\right\Vert_{L^2\left(0,x_0\right)}\left\Vert\partial_t u\right\Vert_{L^2\left(0,x_0\right)}.
\end{split}
    \]}
    
\noindent
Therefore, taking the $H^{1}\left(0,x_{0}\right)$ norm of the identity
{in} \eqref{eq:u_IBVP_ident}, with $u_{\infty}$ defined as in~\eqref{eq:u_infty_1D},
we deduce from 
\eqref{eq:u_bvp_bnds} that
	\begin{align*}
	\left\Vert u\left(\cdot,t\right)-u_{\infty}\right\Vert
          _{H^{1}\left(0,x_0\right)}
          \leq
          &\frac{1}{2}\left\Vert u\left(x,t\right)-u\left(0,t\right)\right\Vert_{H^{1}\left(0,x_0\right)}+\frac{1}{2}\left\Vert u\left(x,t\right)-u\left(x_0,t\right)\right\Vert_{H^{1}\left(0,x_0\right)}\\
          	&+{\frac{1}{2}\sqrt{\frac{x_0}{\alpha_0\beta_0}}\left\Vert\beta\right\Vert_{L^2\left(0,x_0\right)}\left\Vert\partial_t u\right\Vert_{L^2\left(0,x_0\right)}}\\
          \leq
          &\sqrt{1+x_0^2}\left\Vert\partial_x u\right\Vert_{L^2\left(0,x_0\right)}+\frac{1}{2}\sqrt{\frac{x_0}{\alpha_0\beta_0}}\left\Vert\beta\right\Vert_{L^2\left(0,x_0\right)}\left\Vert\partial_t u\right\Vert_{L^2\left(0,x_0\right)}.
	\end{align*}
	Finally, employing \eqref{eq:u_x_t_bnds}{,} we obtain
	\begin{align*}
	\left\Vert u\left(\cdot,t\right)-u_{\infty}\right\Vert _{H^{1}\left(0,x_{0}\right)}\leq&\frac{1}{2 \left(C_{0}t_0\right)^{1/2}}\left[2\left(\left(1+x_{0}^2\right)\left\Vert 1/\alpha\right\Vert _{L^{\infty}\left(0,x_{0}\right)}\right)^{1/2}\right.\\
	&+\left.\left(\frac{x_0}{\alpha_0\beta_0}\left\Vert 1/\beta\right\Vert _{L^{\infty}\left(0,x_{0}\right)}\right)^{1/2}\left\Vert \beta\right\Vert _{L^{2}\left(0,x_{0}\right)}\right]
	 {\widetilde{E}_{loc}[u]^{1/2}}\left(t\right)
	\end{align*}
	which, together with~\eqref{ineq:tE-decay}, leads to the
          bound of $\left\Vert
            u\left(\cdot,t\right)-u_{\infty}\right\Vert
          _{H^{1}\left(0,x_{0}\right)}$ in
        \eqref{eq:u_exp_conv}. 
{From~\eqref{eq:lmb0_simpl1} and~\eqref{ineq:tE-decay}, we find the claimed decay rate {$\Lambda=\frac{\gamma_{0}}{2 \, t_0}e^{-2\gamma_{0}}\left(1-4\gamma_{0}e^{-4\gamma_{0}}\right)^{1/2}$.}}   

Since $\left\Vert
          u\left(\cdot,t\right)-u_{\infty}\right\Vert
        _{H^{1}\left(0,x_{0}\right)}$ decays to zero, $u_\infty$ is
        uniquely defined as in~\eqref{eq:u_infty_1D}. This
          concludes the proof. \qed

        {
          In the proof of Theorem \ref{thm:1D_decay},
modified weight functions could also have been obtained by imposing the
initial condition for \eqref{eq:phi_syst} at an internal point $y_0$ in
$(0,1)$ instead of at $y_0=0$, and then solving the two  corresponding linear ODEs forward and backward. 
Potentially, this may lead to 
an increased maximal decay rate,
as compared to that we have deduced from~\eqref{eq:tanh_eq}.
}
       
We also remark that the decay rate given by $\lambda_0(\gamma_0)$ is
exponentially small for large~$\gamma_0$. This is not an artefact of
the approximation of $\lambda_*(\gamma_0)$ by $\lambda_0(\gamma_0)$
(see Figure~\ref{fig:lmb_vs_gam}), but is rather characteristic to the weight function approach in the modified local energy.


\section{Proof of Theorem \ref{thm:better_rate}}\label{sec:better_rate_proof}
As in the proof of Theorem~\ref{thm:1D_decay} in Section~\ref{sec:1D_decay_proof}, we assume, without loss of generality, that $\Omega=\Omega_{in}=\Omega_0=\left(0,x_0\right)$ for some $x_0>0$. 	
	
	To prove the claim, we shall modify Step 1 of the proof of Theorem \ref{thm:1D_decay}. 
	Namely, instead of working with the modified local energy
        functional \eqref{eq:E_loc_mod}, we shall focus on its
        conventional counterpart \eqref{eq:E_loc} and estimate the quantities that enter it. 
        {Since we shall not
          differentiate the energy functional in time, we 
          work here directly with the mild solution $u$.}
        {We consider again equation~\eqref{eq:v_pbm} in the
          variables~$(y,\tau)$ given by~\eqref{eq:y_def}--\eqref{eq:tau_def}.
By treating the third term on the left-hand side as a (known)
          source term, it results in a wave equation with constant
          coefficients, to which we} 
apply the d'Alembert  formula:
	\begin{equation}
	v\left(y,\tau\right)=\frac{1}{2}\left[v_{0}\left(y+\tau\right)+v_{0}\left(y-\tau\right)\right]+\frac{1}{2}\int_{y-\tau}^{y+\tau}v_{1}\left(z\right)dz-\frac{1}{2}\int_{0}^{\tau}\int_{y-s}^{y+s}\gamma\left(z\right)\partial_{y}v\left(z,\tau-s\right)dz ds.\label{eq:dAlemb}
	\end{equation}
	Therefore, we obtain  
	\begin{align}
	\partial_{y}v\left(y,\tau\right)= & \frac{1}{2}\left[v_{0}^{\prime}\left(y+\tau\right)+v_{0}^{\prime}\left(y-\tau\right)\right]+\frac{1}{2}\left[v_{1}\left(y+\tau\right)-v_{1}\left(y-\tau\right)\right]\nonumber \\
	& +\frac{1}{2}\int_{0}^{\tau}\left[\gamma\left(y-s\right)\partial_{y}v\left(y-s,\tau-s\right)-\gamma\left(y+s\right)\partial_{y}v\left(y+s,\tau-s\right)\right]ds,\label{eq:dv_dy}
	\end{align}
	\begin{align}
	\partial_{\tau}v\left(y,\tau\right)= & \frac{1}{2}\left[v_{0}^{\prime}\left(y+\tau\right)-v_{0}^{\prime}\left(y-\tau\right)\right]+\frac{1}{2}\left[v_{1}\left(y+\tau\right)+v_{1}\left(y-\tau\right)\right]\nonumber \\
	& -\frac{1}{2}\int_{0}^{\tau}\left[\gamma\left(y-s\right)\partial_{y}v\left(y-s,\tau-s\right)+\gamma\left(y+s\right)\partial_{y}v\left(y+s,\tau-s\right)\right]ds.\label{eq:dv_dt}
	\end{align}
        Identity~\eqref{eq:dv_dt} becomes evident 
by making a change of variable $s\mapsto \tau-s$ in the
        double-integral term in~\eqref{eq:dAlemb} 
        before and after the differentiation with respect to $\tau$: 
\begin{align*}
&\partial_\tau\int_{0}^{\tau}\int_{y-s}^{y+s}\gamma\left(z\right)\partial_{y}v\left(z,\tau-s\right)dz ds=\partial_\tau\int_{0}^{\tau}\int_{y-\tau+s}^{y+\tau-s}\gamma\left(z\right)\partial_{y}v\left(z,s\right)dz ds\\
&=\int_{0}^{\tau}\left[\gamma\left(y+\tau-s\right)\partial_{y}v\left(y+\tau-s,s\right)+\gamma\left(y-\tau+s\right)\partial_{y}v\left(y-\tau+s,s\right)\right]ds\\
&=\int_{0}^{\tau}\left[\gamma\left(y+s\right)\partial_{y}v\left(y+s,\tau-s\right)+\gamma\left(y-s\right)\partial_{y}v\left(y-s,\tau-s\right)\right]ds.
\end{align*}

	Note that{, for $\tau\ge1$,} the 
        terms in \eqref{eq:dv_dy} and \eqref{eq:dv_dt}
{that do not contain the integral}
vanish 
{for all $y$ in} $\left(0,1\right)$. 
	To estimate $\mathcal{E}_{loc}$ for $\tau\ge1$, we
        insert~\eqref{eq:dv_dy} and~\eqref{eq:dv_dt} in {the
          definition of $\mathcal{E}_{loc}\left(\tau\right)$ in}~\eqref{eq:E_loc} and use that $\textrm{supp }\gamma\subset\left(0,1\right)$:
	\begin{align}
	\label{eq:E_int_estim}\mathcal{E}_{loc}\left(\tau\right)= & \frac{1}{4}\int_{0}^{1}\left[\left(\int_{y}^{1}\gamma\left(r\right)\partial_{y}v\left(r,\tau+y-r\right)dr\right)^{2}\right.\\
	&\left.+\left(\int_{0}^{y}\gamma\left(r\right)\partial_{y}v\left(r,\tau-y+r\right)dr\right)^{2} \right]\frac{dy}{\rho\left(y\right)}\nonumber\\
	\leq & \frac{1}{4}\int_{0}^{1}\left[\int_{y}^{1}\rho\left(\tilde{r}\right)\gamma^{2}\left(\tilde{r}\right)d\tilde{r}\int_{y}^{1}\left(\partial_{y}v\left(r,\tau+y-r\right)\right)^{2}\frac{dr}{\rho\left(r\right)}\right.\nonumber\\
&\left.+\int_{0}^{y}\rho\left(\tilde{r}\right)\gamma^{2}\left(\tilde{r}\right)d\tilde{r}\int_{0}^{y}\left(\partial_{y}v\left(r,\tau-y+r\right)\right)^{2}\frac{dr}{\rho\left(r\right)}                                                                          \right]\frac{dy}{\rho\left(y\right)}.\nonumber
        \end{align}
Denoting
	\begin{equation*}
	b_{0}:=\left\Vert \alpha\right\Vert
        _{L^{\infty}\left(0,x_{0}\right)}^{3/2}\left\Vert
          \beta\right\Vert
        _{L^{\infty}\left(0,x_{0}\right)}^{1/2}\left\Vert
          \left(1/\left(\alpha\beta\right)^{1/2}\right)^{\prime}\right\Vert
        _{L^{2}\left(0,x_{0}\right)}^{2},
	\end{equation*}
{as in}~\eqref{eq:b0_cond}, we have 
\begin{equation*}
\max\left\{\sup_{y\in\left(0,1\right)}\frac{1}{\rho\left(y\right)}\int_{0}^{y}\rho\left(s\right)\gamma^{2}\left(s\right)ds,\sup_{y\in\left(0,1\right)}\frac{1}{\rho\left(y\right)}\int_{y}^{1}\rho\left(s\right)\gamma^{2}\left(s\right)ds\right\}\leq b_0t_0.
\end{equation*}
Therefore, we continue with~\eqref{eq:E_int_estim}, and estimate
\begin{align}\label{eq:E_loc_ineq_estim}
	\mathcal{E}_{loc}\left(\tau\right) &\leq\frac{b_{0}t_0}{4}\left[\int_{0}^{1}\int_{0}^{y}\left(\partial_{y}v\left(r,\tau-y+r\right)\right)^{2}\frac{1}{\rho\left(r\right)}drdy\right.\\
 &\phantom{=}\qquad\quad\left.+\int_{0}^{1}\int_{y}^{1}\left(\partial_{y}v\left(r,\tau+y-r\right)\right)^{2}\frac{1}{\rho\left(r\right)}drdy\right].\nonumber
\end{align}
{In the two terms on the right-hand side
  of~\eqref{eq:E_loc_ineq_estim}, the integration involves the values of $\partial_y v$ on two right-angled triangles with vertices $\left\{\left(0,\tau-1\right),\left(0,\tau\right),\left(1,\tau\right)\right\}$ and $\left\{\left(1,\tau-1\right),\left(0,\tau\right),\left(1,\tau\right)\right\}$. These triangles are parametrised by the sets of lines parallel to their hypotenuses, which are the two characteristic families of the wave equation in \eqref{eq:v_pbm}. We now change the integration variables to have parametrisation of these triangles by the sets of horizontal lines. Consequently, for the first integral, denoting $\sigma:=\tau-y+r$, we have that
  $\sigma\in(\tau-1,\tau)$ and $r\in(0,\sigma-\tau+1)$. Similarly, for
  the second integral, denoting $\sigma:=\tau+y-r$, we have that
  $\sigma\in(\tau-1,\tau)$ and $r\in(\tau-\sigma,1)$. Therefore, rewriting  
  the integrals in~\eqref{eq:E_loc_ineq_estim} as integrals in the variables
  $\sigma$ and $r$, we obtain
  \[
    \begin{split}
     & \int_{0}^{1}\int_{0}^{y}\left(\partial_{y}v\left(r,\tau-y+r\right)\right)^{2}\frac{1}{\rho\left(r\right)}drdy
     +\int_{0}^{1}\int_{y}^{1}\left(\partial_{y}v\left(r,\tau+y-r\right)\right)^{2}\frac{1}{\rho\left(r\right)}drdy\\
     &\qquad=
     \int_{\tau-1}^{\tau}\int_{0}^{\sigma-\tau+1}\left(\partial_{y}v\left(r,\sigma\right)\right)^{2}\frac{1}{\rho\left(r\right)}drd\sigma+\int_{\tau-1}^{\tau}\int_{\tau-\sigma}^{1}\left(\partial_{y}v\left(r,\sigma\right)\right)^{2}\frac{1}{\rho\left(r\right)}drd\sigma.
      \end{split}
    \]
With this identity, taking into account the positivity of the
integrands, we obtain from~\eqref{eq:E_loc_ineq_estim} that
}
\begin{align}\label{eq:E_loc_ineq_estim1}  
  \mathcal{E}_{loc}\left(\tau\right)
    \leq\frac{b_{0}t_0}{2}\int_{\tau-1}^{\tau}\int_{0}^{1}\left(\partial_{y}v\left(r,\sigma
\right)\right)^{2}\frac{1}{\rho\left(r\right)}drd\sigma
      \leq
  b_{0}t_0\int_{\tau-1}^{\tau}\mathcal{E}_{loc}\left(\sigma\right)d\sigma.
\end{align}
Since $\mathcal{E}_{loc}\left(\tau\right)$ is a positive non-increasing
function (recall \eqref{eq:E_aux_dt}), it follows
from~\eqref{eq:E_loc_ineq_estim1} that
	\[
	\mathcal{E}_{loc}\left(\tau\right)\leq b_{0}t_0\,\mathcal{E}_{loc}\left(\tau-1\right),\hspace{1em}\tau\ge1.
	\]
	Iterating this inequality $\left\lfloor \tau\right\rfloor $
	times (where $\left\lfloor X\right\rfloor $ denotes the integer part
	of $X$), {and using the assumption $b_{0}t_0<1$,} we obtain, for $\tau \geq 1$,
	\begin{eqnarray*}
	\mathcal{E}_{loc}\left(\tau\right)&\leq&
\mathcal{E}_{loc}\left(\lfloor\tau\rfloor\right)\leq
 (b_{0}t_0)^2\mathcal{E}_{loc}\left(\lfloor\tau\rfloor-2\right)\leq\ldots
\leq {(b_{0}t_0)^{\lfloor\tau\rfloor}\mathcal{E}_{loc}\left(0\right)}\nonumber\\
  &=&  \mathcal{E}_{loc}\left(0\right)e^{\lfloor\tau\rfloor\log (b_{0}t_0)} \leq
	 \mathcal{E}_{loc}\left(0\right)e^{(\tau-1)\log (b_{0}t_0)}
	=\mathcal{E}_{loc}\left(0\right)\frac{e^{\tau\log (b_{0}t_0)}}{b_0t_0},\label{eq:E_loc_dec_alt}
	\end{eqnarray*}
	which, due to 
        {$\log (b_{0}t_0)<0$}, shows an exponential decay {in $\tau$}.
	
	When scaling back to the original variables $(x,t)$ we obtain 
\[
  \tilde{\mathcal{E}}_{loc}\left(t\right) := 
  \mathcal{E}_{loc}\left(\frac{t}{t_0}\right) 
  \le \tilde{\mathcal{E}}_{loc}\left(0\right)\frac{e^{\frac{t}{t_0}\log (b_{0}t_0)}}{b_0t_0}, \hspace{1em}t\geq t_0.
\]
	Step 3 of the proof of Theorem \ref{thm:1D_decay} can then be repeated (with a slight modification
	of the multiplicative constants), yielding the exponential convergence (with rate \eqref{Lambda})
	of the solution to the constant $u_{\infty}$.

{We note that, even though the last bound was proved for $t\geq t_0$, it can be extended to $t\geq 0$ merely at the expense of enlarging the constant in front of the time-decaying exponential factor. This is due to the well-posedness result for the wave equation on a finite time-interval, see Proposition \ref{prop:WP} and, in particular, \eqref{eq:WP_estim}.}
\qed

\section{Conclusion}
\label{sec:conc}

For the homogeneous wave equation in $\mathbb R$ with 
compactly supported initial data, in the case of constant coefficients,
the solution converges to the
steady state in finite time. However, exponential
convergence is the generic scenario in the case of
variable coefficients.

We proved two different results for
this convergence in local energy. The first one employs
a multiplier technique, constructing appropriate weight functions
to define a modified local energy that satisfies a first order
differential inequality. This allowed us to prove
the exponential decay. 
Using the solution balance over space-time rectangles, we were
able to identify the constant steady state.
The second result improves the first one, under an additional
  assumption on the variation of the coefficients.
After rescaling the equation, its 
integral form allowed us to estimate
the local energy by iterating in time, leading to an
exponential decay with improved rates. 

Our results complement those of algebraic time-decay
presented in~\cite{Charao}, where 
  the initial data were assumed to be
  sufficiently localised but not compactly supported. In that case, exponential decay is generally not expected.   

\appendix
\renewcommand{\thesection}{}
\section{}
\renewcommand{\thesection}{\Alph{section}}
\begin{lem}\label{lem:tanh_eq}
For any fixed $\gamma_0>0$, the equation
\begin{equation}\label{eq:tanh_eq2}
{\tanh\left({\sqrt{\gamma_{0}^{2}+\lambda^{2}}}\right)=\frac{\sqrt{\gamma_0^2+\lambda^2}}{\gamma_0+\lambda}, \quad \lambda > 0,}
\end{equation}
{has a unique solution $\lambda=\lambda_{*}$.}
Moreover, for $\lambda>0$, the strict inequality \eqref{eq:tanh_ineq1} 
{is satisfied if and only if $\lambda\in(0,\lambda_*)$}.
\end{lem}
\begin{proof}
Since the right-hand side of \eqref{eq:tanh_eq2}, for {$\lambda>0$}, is smaller than 1, equation~\eqref{eq:tanh_eq2} is equivalent to 
\begin{equation}\label{eq:f_gam_def}
  f_{\gamma_0}(\tilde\lambda):=\gamma_{0}\sqrt{1+\tilde\lambda^{2}}
  -\artanh\left({\frac{\sqrt{1+\tilde\lambda^2}}{1+\tilde\lambda}}\right)=0,
\end{equation}
where {$\tilde\lambda:=\lambda/\gamma_0$}. We have
\[
  f_{\gamma_0}'(\tilde\lambda)=\frac{2\gamma_0\tilde\lambda^2-\tilde\lambda+1}{2\tilde\lambda \sqrt{1+\tilde\lambda^{2}}},\qquad
  \lim_{\tilde\lambda\to +\infty} f'_{\gamma_0}(\tilde\lambda) = \gamma_0>0.
\]
Hence {$\lim_{\tilde\lambda\to
    +\infty}f_{\gamma_0}(\tilde\lambda)=+\infty$}
and, moreover, {$\lim_{\tilde\lambda\to 0^+}f_{\gamma_0}(\tilde\lambda)=-\infty$.}
\\

\noindent
{\bf Case 1:} 
For $\gamma_0\ge1/8$, {we have} 
$f'_{\gamma_0}(\tilde\lambda)\ge0$ for any $\tilde\lambda>0$.
Hence, $f_{\gamma_0}$ is strictly increasing and $f_{\gamma_0}$ has a
unique zero at some $\tilde\lambda_*>0$ and
$f_{\gamma_0}(\tilde\lambda)<0$ for
$\tilde\lambda\in(0,\tilde\lambda_*)$.

\noindent
{\bf Case 2:} 
For $\gamma_0\in(0,1/8)$, we consider $f_{\gamma_0}(\tilde\lambda_1)$,
where $\tilde\lambda_1 := \frac{1-\sqrt{1-8\gamma_0}}{4\gamma_0}$ {is} the
first zero of $f_{\gamma_0}'$.
{Define
  $g(\gamma_0):=\frac{1-\sqrt{1-8\gamma_0}}{4\gamma_0}$. From
  $g^\prime(\gamma_0)>0$ for all $\gamma_0\in (0,1/8)$,
  $\lim_{\gamma_0\to 0^+}g(\gamma_0)=1$, and $g(1/8)=2$, we deduce that
$g$ is strictly increasing and $\tilde{\lambda}_1=g(\gamma_0)\in(1,2)$ for all
$\gamma_0\in(0,1/8)$. Since
$\artanh\left({\frac{\sqrt{1+\tilde\lambda^2}}{1+\tilde\lambda}}\right)$
is an increasing function of $\tilde\lambda$ in $(1,2)$, we obtain
from~\eqref{eq:f_gam_def} that
$f_{\gamma_0}(\tilde\lambda_1)\le \frac{\sqrt{1+4}}{8}-\artanh
\left(\frac{\sqrt{2}}{2}\right)\simeq -0.6019$.}
{This}
shows that $f_{\gamma_0}(\tilde\lambda_1)<0$
for $\gamma_0\in(0,1/8)$. Hence, $f_{\gamma_0}$ also has a unique zero
$\tilde\lambda_*{>\tilde\lambda_2}>0$ in this case{, where
  $\tilde\lambda_2:=\frac{1+\sqrt{1-8\gamma_0}}{4\gamma_0}$ is the
  second zero of $f_{\gamma_0}'$}.
\medskip

{This proves unique solvability of equation~\eqref{eq:tanh_eq2}. Moreover,
  the reasoning above also shows that, for any $\gamma_0>0$, there is
  a unique $\lambda_{*}=\lambda_{*}\left(\gamma_0\right)>0$ such that, for $\lambda>0$, the following equivalences hold true:
\begin{align*}
\lambda \in (0,\lambda_{*}) \quad & \Longleftrightarrow \quad f_{\gamma_0}\left(\lambda/\gamma_0\right) < 0 \quad \Longleftrightarrow \quad \sqrt{\gamma_{0}^2+\lambda^{2}}
  -\artanh\left({\frac{\sqrt{\gamma_{0}^2+\lambda^2}}{\gamma_{0}^2+\lambda}}\right) < 0 \\
  & \Longleftrightarrow \quad \tanh\left({\sqrt{\gamma_{0}^{2}+\lambda^{2}}}\right) < \frac{\sqrt{\gamma_0^2+\lambda^2}}{\gamma_0+\lambda}.
\end{align*}
Here, the last equivalence is due to the strict monotonicity of the $\tanh$ function. The second part of the statement thus follows.
}
\end{proof}


\begin{lem}\label{lem:app_ineq}
For any $a>0$ and 
{
\begin{equation}\label{q-assump}
1<q<\sqrt{1+\dfrac{e^{4a}+4a-e^{2a}\sqrt{e^{4a}+8a}}{8a^{2}}},
\end{equation}
}
we have
	\begin{equation}
	\frac{q}{\tanh\left(a q\right)}>1+\sqrt{q^{2}-1}.\label{eq:q_cond}
	\end{equation}

\end{lem}
\begin{proof}
{The proof is split into two steps.}
\\
{\bf Step 1.} We shall first analyse for which values $q_0\ge1$ the following auxiliary inequality holds:
	\begin{equation}
	\left(1+2a\right)e^{-2a}-2a e^{-2a}q_{0}^2\geq\sqrt{q_{0}^{2}-1}.\label{eq:q0_cond}
	\end{equation}

	Using $Q:=q_{0}^{2}\ge1$ we rewrite~\eqref{eq:q0_cond}
        as
	\begin{equation}
	1+2a-2a Q\geq e^{2a}\sqrt{Q-1}.\label{eq:Q_cond}
	\end{equation}
	This is equivalent to
	\begin{equation}
	4a^{2}Q^{2}-\left(4a\left(1+2a\right)+e^{4a}\right)Q+\left(1+2a\right)^{2}+e^{4a}\geq0,\label{eq:Q_cond_equiv}
	\end{equation}
	as long as both sides of \eqref{eq:Q_cond} are nonnegative, i.e.\ for $Q\in\left[1,\frac{1+2a}{2a}\right]$.
	
The left-hand side of  \eqref{eq:Q_cond_equiv} is a quadratic expression in $Q$ which is positive 
	for $Q\in\mathbb{R}\backslash\left[Q_{1},Q_{2}\right]$, with 
	\begin{align*}
	&Q_{1}:=Q_{0}-\frac{e^{2a}\sqrt{e^{4a}+8a}}{8a^{2}},\hspace{1em}Q_{2}:=Q_{0}+\frac{e^{2a}\sqrt{e^{4a}+8a}}{8a^{2}},&\\
	&Q_{0}:=1+\frac{e^{4a}+4a}{8a^{2}}.&
	\end{align*}	
	It can be checked easily that $Q_{1}\in\left(1,\frac{1+2a}{2a}\right)$, 
	and therefore any $Q\in\left[1,Q_{1}\right]$ satisfies~\eqref{eq:Q_cond}.
	Consequently, 
	{\eqref{eq:q0_cond} is satisfied if and only if $q_{0}\in\left[1,\sqrt{Q_{1}}\right]$.\\
	}

\noindent
{\bf Step 2.} We observe that since $e^{-bx}\geq e^{-b}\left(1+b-bx\right)$ for any $b,\,x\in\mR$ 
(by convexity of $e^{-bx}$), we have (with {$b:=2a>0$, $x:=q_0>1$})
	\begin{align}
	\frac{q_{0}}{\tanh\left(a q_{0}\right)}=&\frac{q_{0}\left(e^{a q_{0}}+e^{-aq_{0}}\right)}{e^{aq_{0}}-e^{-a q_{0}}}>q_{0}\left(1+e^{-2a q_{0}}\right)\nonumber\\
	\geq&q_{0}\left[1+e^{-2a}\left(1+2a-2aq_{0}\right)\right]\label{step2} \\ 
	\ge&1+\left(1+2a\right)e^{-2a}-2a e^{-2a}q_{0}^2.\nonumber
	\end{align}
{By assumption \eqref{q-assump}, we have $q\in\left(1,\sqrt{Q_{1}}\right)$. Therefore, \eqref{step2} can be combined with \eqref{eq:q0_cond}, to yield \eqref{eq:q_cond}.
}
\end{proof}


\begin{lem}\label{lem:A2}
The function 
\begin{equation}
{\lambda_0(\gamma_0):=\gamma_{0}e^{-2\gamma_{0}}\sqrt{1-4\gamma_{0}e^{-4\gamma_{0}}},\quad \gamma_0>0}\label{eq:lmb0_simpl}
\end{equation}
satisfies $\lambda_0(\gamma_0)<\lambda_*(\gamma_0)$, where $\lambda_*(\gamma_0)$ is defined via \eqref{eq:tanh_eq}. Moreover, each $\lambda=\lambda_0(\gamma_0)$ satisfies  inequality \eqref{eq:tanh_ineq1}. 
\end{lem}

\begin{proof}
To show that $\lambda{:=}\lambda_0(\gamma_0)$ satisfies
\eqref{eq:tanh_ineq1} we first derive the following auxiliary
inequality:
\begin{align*}
\sqrt{1+\left(\lambda/\gamma_{0}\right)^{2}}
&= \sqrt{1+e^{-4\gamma_{0}}\left(1-4\gamma_{0}e^{-4\gamma_{0}}\right)} \nonumber\\
&< \sqrt{1+\dfrac{e^{4\gamma_{0}}+4\gamma_{0}-e^{2\gamma_{0}}\sqrt{e^{4\gamma_{0}}+8\gamma_{0}}}{8\gamma_{0}^2}},
\end{align*}
where we employed definition \eqref{eq:lmb0_simpl} and used the estimate 
$\sqrt{1+x}<1+x/2-x^{2}/8+x^{3}/16$, with $x=8\gamma_0e^{-4\gamma_0}>0$, which follows from the Taylor expansion of the square root. 

{The above inequality forms} exactly
the {assumption \eqref{q-assump}}
with $q:=\sqrt{1+\left(\lambda/\gamma_{0}\right)^{2}}>1$ and
$a:=\gamma_{0}$. The assertion of {Lemma~\ref{lem:app_ineq}}
is equivalent to the validity of \eqref{eq:tanh_ineq1}.

{Finally, Lemma \ref{lem:tanh_eq} yields 
$\lambda_0(\gamma_0)<\lambda_*(\gamma_0)$, which completes the present proof.}
\end{proof}

\bibliographystyle{abbrv}
\bibliography{limampl_vs_decay}
\bigskip
Email address: \href{mailto:anton.arnold@asc.tuwien.ac.at}{anton.arnold@asc.tuwien.ac.at}\\
Email address: \href{mailto:sjoerd.geevers@univie.ac.at}{sjoerd.geevers@univie.ac.at}\\
Email address: \href{mailto:ilaria.perugia@univie.ac.at}{ilaria.perugia@univie.ac.at}\\
Email address: \href{mailto:dmitry.ponomarev@asc.tuwien.ac.at}{dmitry.ponomarev@asc.tuwien.ac.at}

\end{document}